\documentclass[11pt]{amsart}

\usepackage{amssymb}

\newtheorem{thrm}{Theorem}
\newtheorem{cor}[thrm]{Corollary}
\newtheorem{defin}[thrm]{Definition}
\newtheorem{lem}[thrm]{Lemma}
\newtheorem*{thmA}{Theorem~A}
\newtheorem*{thmB}{Corollary~A}
\newtheorem*{thmC}{Corollary~B}
\usepackage{cite}
\usepackage{fullpage}
\usepackage{color}
\usepackage{pdfsync}

\newcommand{\Sym}{\mathop{\mathrm{Sym}}}
\def\Z{{\mathbb Z}}
\def\Aut{{\rm Aut}}
\def\la{{\langle}}
\def\ra{{\rangle}}
\def\cal{\mathcal}
\def\gcd{{\rm gcd}}

\def\fix{{\rm fix}}

\def\Stab{{\rm Stab}}

\def\tl{\,{\triangleleft}\,}

\def\mod{{\rm mod\ }}

\def\AGL{{\rm AGL}}

\def\GL{{\rm GL}}
\def\Alt{{\rm Alt}}

\usepackage{hyperref}

\begin{document}
\pagestyle{plain}

\title{CI-groups with respect to ternary
  relational structures: new examples}

\bigskip
\maketitle

\begin{center}
\begin{tabular}{c c}
{\sc Edward Dobson} & {\sc Pablo Spiga}\\
%                     &                      \\
{\small Department of Mathematics and Statistics} & {\small Dipartimento di Matematica Pura e Applicata} \\
{\small Mississipi State University}   & {\small Universit\`a degli Studi di Milano-Bicocca} \\
{\small PO Drawer MA} &{\small Via Cozzi 53} \\
{\small Mississipi State, MS 39762 USA} & {\small 20126 Milano, Italy}\\
{\small\tt dobson@math.msstate.edu} & {\small\tt pablo.spiga@unimib.it}
\end{tabular}
\end{center}

\begin{abstract}
We find a sufficient condition to establish that certain abelian
groups are not CI-groups with respect to ternary relational
structures, and then show that the groups $\Z_3\times\Z_2^2$,
$\Z_7\times\Z_2^3$, and $\Z_5\times\Z_2^4$ satisfy this condition.  Then we completely determine  which groups $\Z_2^3\times\Z_p$, $p$ a prime, are CI-groups with respect to binary and ternary relational structures.  Finally, we show that $\Z_2^5$ is not a CI-group with respect to ternary relational structures.
\end{abstract}

\section{Introduction}\label{intro}
In recent years, there has been considerable interest in which
groups $G$ have the property that any two Cayley graphs of $G$ are
isomorphic if and only if they are isomorphic by a group
automorphism of $G$.  Such a group is a called a  CI-group with respect to
graphs, and this problem is often referred to as the Cayley
isomorphism problem.   The
interested reader is referred to \cite{Li2002} for a
survey on CI-groups
with respect to graphs.  Of course, the Cayley isomorphism problem
can and has been considered for other types of combinatorial
objects.  Perhaps the most significant such result is a well-known
theorem of P\'alfy \cite{Palfy1987} which states that a group $G$ of order
$n$ is a CI-group with respect to every class of combinatorial
objects if and only if $n = 4$ or $\gcd(n,\varphi(n)) = 1$, where
$\varphi$ is the Euler phi function.  In fact, in proving this result, P\'alfy
showed that if a group $G$ is not a CI-group with respect to some
class of combinatorial objects, then $G$ is not a CI-group with
respect to quaternary relational structures.  As much work has been
done on the case of binary relational structures (i.e., digraphs),
until recently there was a ``gap" in our knowledge of the Cayley
isomorphism problem for $k$-ary relational structures with
$k = 3$.  As additional motivation to study this problem, we remark
that  a group $G$ that is a CI-group
with respect to ternary relational structures is necessarily a
CI-group with respect to binary relational structures.

Although Babai~\cite{Babai1977} showed in~$1977$ that the dihedral group of order $2p$
is a CI-group with respect to ternary relational structures,
no additional work was done on this problem until the first author
considered the problem in $2003$
\cite{Dobson2003}. Indeed, in \cite{Dobson2003} a relatively short list
of groups is given and it is proved that every
CI-group with respect to ternary
relational structures lies in this list (although not every group in this list is necessarily a CI-group with respect to ternary relational structures). Additionally,
several groups in the list were shown to be CI-groups with respect
to ternary relational structures.  Recently, the second author \cite{Spiga2008} has
shown that two groups given in \cite{Dobson2003} are not
CI-groups with respect to ternary relational structures, namely
$\Z_3\ltimes Q_8$ and $\Z_3\times Q_8$.  In this paper, we give a
sufficient condition to ensure that certain abelian groups are
not CI-groups with respect to ternary relational structures (Theorem \ref{main}), and then
show that $\Z_3\times\Z_2^2$, $\Z_7\times\Z_2^3$,
and $\Z_5\times\Z_2^4$ satisfy this
condition in Corollary~\ref{coro1} (and so are not CI-groups with respect to ternary
relational structures).  We then show that $\Z_5\times \Z_2^3$ is a CI-group with respect to ternary relational structures.  As the first author has shown
\cite{Dobson2010a} that $\Z_2^3\times\Z_p$ is a CI-group with respect to
ternary relational structures provided that $p\ge 11$, we then have a complete determination of which groups $\Z_2^3\times\Z_p$, $p$ a prime, are CI-groups with respect to ternary relational structures.

\begin{thmA}
The group $\Z_2^3\times\Z_p$ is a CI-group with respect to color ternary relational structures if and only if $p\not = 3$ and $7$.
\end{thmA}
We will show that both $\Z_2^3\times\Z_3$ and $\Z_2^3\times\Z_7$ are CI-groups with respect to binary relational structures. As it is already known that $\Z_2^4$ is a CI-group with respect to binary relational structures~\cite{Li2002}, we have the following result.
\begin{thmB}
The group $\Z_2^3\times\Z_p$ is a CI-group with respect to color binary relational structures for all primes $p$.
\end{thmB}
We are then left in the situation of knowing whether or not any subgroup of $\Z_2^3\times\Z_p$ is a CI-group with respect to binary or ternary relational structures, with the exception of $\Z_2^2\times\Z_7$ with respect to ternary relational structures (as $\Z_2^2\times\Z_7$ is a CI-group with respect to binary relational structures \cite{KovacsM2009}).  We show that $\Z_2^2\times\Z_7$ is a CI-group with respect to ternary relational structures (which generalizes a special case of the main result of \cite{KovacsM2009}) and we prove the following.

\begin{thmC}
The group $\Z_2^2\times\Z_p$ is a CI-group with respect to color ternary relational structures if and only if $p\neq 3$.
\end{thmC}

Finally, using Magma~\cite{Magma} and GAP~\cite{GAP}, we show that $\Z_2^5$ is not a CI-group with respect to ternary relational structures.

We conclude this introductory section with the formal definition of the objects we are interested in.
\begin{defin}{\rm
A {\it $k$-ary relational structure} is an ordered pair $X =
(V,E)$, with $V$ a set and $E$ a subset of $V^k$.
Furthermore, a {\it color $k$-ary relational structure} is an ordered pair $X =
(V,(E_1,\ldots,E_c))$, with $V$ a set and $E_1,\ldots,E_c$ pairwise disjoint subsets of $V^k$.  If $k = 2,3$, or $4$, we simply say that $X$ is a (color) binary, ternary, or quaternary relational structure.}
\end{defin}

The following two definitions are due to Babai~\cite{Babai1977}.

\begin{defin}\label{def:2}{\rm For a group $G$, define $g_L:G\rightarrow G$ by $g_L(h)
  = gh$, and let $G_L = \{g_L:g\in G\}$.  Then
  $G_L$ is a permutation group on $G$, called the {\it left regular representation of
    $G$}.  We will say that a (color) $k$-ary relational structure $X$ is a {\it
   Cayley (color) $k$-ary relational structure of $G$} if $G_L\le\Aut(X)$ (note that this  implies $V = G$).  In general, a combinatorial object $X$ will be called a
  {\it Cayley object of $G$} if $G_L\le\Aut(X)$.}
\end{defin}

\begin{defin}\label{def:3}{\rm
For a class ${\cal C}$ of Cayley objects of $G$, we say that $G$ is
a {\it CI-group with respect to ${\cal C}$} if whenever $X,Y\in{\cal
C}$, then $X$ and $Y$ are isomorphic if
and only if they are isomorphic by a group automorphism of $G$.}
\end{defin}

It is clear that if $G$ is a CI-group with respect to {\em color} $k$-ary relational structures, then $G$ is a CI-group with respect to $k$-ary relational structures.

\begin{defin}{\rm
For $g,h$ in  $G$, we denote the commutator  $g^{-1}h^{-1}gh$ of $g$
and $h$  by $[g,h]$.}
\end{defin}

\section{The main ingredient and Theorem~A}\label{thm}
We start by proving the main ingredient for our proof of Theorem~A.
\begin{thrm}\label{main}
Let $G$ be an abelian group and $p$ an odd prime. Assume that there
exists an automorphism $\alpha$ of $G$ of order $p$
fixing only the zero element of $G$. Then $\Z_p\times G$ is not a $\mathrm{CI}$-group with respect to color ternary relational structures. Moreover, if there exists a ternary relational structure on $G$ with automorphism group $\la G_L,\alpha\ra$, then $\Z_p\times G$ is not a $\mathrm{CI}$-group with respect to ternary relational structures.
\end{thrm}

\begin{proof}
Since $\alpha$ fixes only the zero element of $G$, we have $\vert G\vert \equiv 1\ (\mod p)$ and so $\gcd (p,\vert G\vert)=1$.

For each $g\in G$, define $\hat{g}:\Z_p\times G\to\Z_p\times G$ by
$\hat{g}(i,j) = (i,j+g)$.  Additionally, define
$\tau,\gamma,\bar{\alpha}:\Z_p\times G\to\Z_p\times G$ by
$\tau(i,j) = (i + 1,j)$, $\gamma(i,j) = (i,\alpha^i(j))$, and
$\bar{\alpha}(i,j) = (i,\alpha(j))$. Then $(\Z_p\times G)_L =
\la\tau,\hat{g}:g\in G\ra$.

Clearly, $\la G_L,\alpha\ra =G_L\rtimes \la \alpha\ra$ is a subgroup of $\Sym(G)$ (where $G_L$ acts on $G$ by  left multiplication and $\alpha$ acts as an automorphism). Note that the stabilizer of $0$ in
$\la G_L,\alpha\ra$ is $\la\alpha\ra$. As $\alpha$ fixes only $0$, we
conclude that for every $g\in G$ with $g\neq 0$, the point-wise stabilizer of $0$ and $g$
in $\langle G_L,\alpha\rangle $ is $1$. Therefore,  by \cite[Theorem 5.12]{Wielandt1969},
there exists a color Cayley ternary relational structure $Z$ of $G$ such that
$\Aut(Z) = \la G_L,\alpha\ra$.  If there exists also a ternary relational structure with automorphism group $\la G_L,\alpha\ra$, then we let $Z$ be one such ternary relational structure.

Let $$U = \{((0_{\Z_p},g),(0_{\Z_p},h)):(0_G,g,h)\in
E(Z)\}\quad \mathrm{and}\quad S =
\{([\hat{g},\gamma](1,0_G),[\hat{g},\gamma](2,0_G)):g\in G\}\cup U$$
and define a (color) ternary relational structure $X$ by $$V(X) = \Z_p\times G\quad
\mathrm{and}\quad E(X) = \{k(0_{\Z_p\times G},s_1,s_2):(s_1,s_2)\in S,k\in(\Z_p\times G)_L\}.$$

\noindent If $Z$ is a color ternary relational structure, then we assign to the edge $k(0_{\Z_p\times G},s_1,s_2)$ the color of the edge $(0_G,g,h)$ in $Z$ if $(s_1,s_2)\in U$ and $(s_1,s_2) = ((0_{\Z_p},g),(0_{\Z_p},h))$, and otherwise we assign a fixed color distinct from those used in $Z$.  By definition of $X$ we have $(\Z_p\times G)_L\le\Aut(X)$ and so $X$ is a (color) Cayley ternary
relational structure of $\Z_p\times G$.

We claim that $\bar{\alpha}\in \Aut(X)$.  As $\bar{\alpha}$ is an
automorphism of $\Z_p\times G$, we have that $\bar{\alpha}\in\Aut(X)$ if and only if
$\bar{\alpha}(S) = S$ and $\bar{\alpha}$ preserves colors (if $X$ is a color ternary relational structure). By definition of $Z$ and $U$, we have $\bar{\alpha}(U) = U$ and $\bar{\alpha}$ preserves colors (if $X$ is a color ternary relational structure). So, it suffices to consider the case $s\in S-U$, i.e., $s=
([\hat{g},\gamma](1,0),[\hat{g},\gamma](2,0))$ for some $g\in G$. Note that now we need not consider colors as all the edges in $S-U$ are of the same color.
Then $\bar{\alpha}\hat{g}(i,j) = (i,\alpha(j) +
\alpha(g)) = \widehat{\alpha(g)}\bar{\alpha}(i,j)$. Thus
$\bar{\alpha}\hat{g} = \widehat{\alpha(g)}\bar{\alpha}$.  Similarly,
$\bar{\alpha}\hat{g}^{-1} = \widehat{\alpha(g)}^{-1}\bar{\alpha}$.  Clearly
$\bar{\alpha}$ commutes with $\gamma$, and so
$\bar{\alpha}[\hat{g},\gamma] = [\widehat{\alpha(g)},\gamma]\bar{\alpha}$. As
$\bar{\alpha}$ fixes $(1,0)$ and $(2,0)$, we see that

\begin{eqnarray*}
\bar{\alpha}(s)=\bar{\alpha}([\hat{g},\gamma](1,0),[\hat{g},\gamma](2,0)) & = &
(\bar{\alpha}[\hat{g},\gamma](1,0),\bar{\alpha}[\hat{g},\gamma](2,0))\\
& =&
([\widehat{\alpha(g)},\gamma]\bar{\alpha}(1,0),
[\widehat{\alpha(g)},\gamma]\bar{\alpha}(2,0))\\
& = &
([\widehat{\alpha(g)},\gamma](1,0),[\widehat{\alpha(g)},\gamma](2,0))\in
(S-U).
\end{eqnarray*}

\noindent Thus $\bar{\alpha}(S) = S$, $\bar{\alpha}$ preserves colors (if $X$ is a color ternary relational structure) and $\bar{\alpha}\in\Aut(X)$.

We claim that $\gamma^{-1}(\mathbb{Z}_p\times G)_L\gamma$ is a
subgroup of $\Aut(X)$. We set $\tau'=\gamma^{-1}\tau\gamma$ and
$g'=\gamma^{-1} \hat{g}\gamma$, for $g\in G$.
Note that $\tau'=\tau \bar{\alpha}^{-1}$. As
$\bar{\alpha}\in\Aut(X)$, we
have that $\tau'\in\Aut(X)$. Therefore it remains to prove that
$\langle g' : g\in G\rangle$ is a subgroup of $\Aut(X)$. Let $e\in
E(X)$ and $g\in G$. Then $e=k((0,0),s)$, where $s\in S$ and
$k=\tau^a\widehat{l}$, for some $a\in\mathbb{Z}_p$, $l\in G$. We have
to prove that $g'(e)\in E(X)$ and has the same color of $e$ (if $X$ is a color ternary relational structure).

Assume that $s\in U$. As $g'(i,j)=(i,j+\alpha^{-i}(g))$, by
definition of $U$, we have $g'[k((0,0),s)]\in E(X)$ and has the same color of $e$ (if $X$ is a color ternary relational structure). So, it
remains to consider the case
$s\in S - U$, i.e., $s=([\widehat{x},\gamma](1,0),[\widehat{x},\gamma](2,0))$ for
some $x\in G$.  As before, we need not concern ourselves with colors because all the edges in $S-U$ are of the same color.

Set $m=k\widehat{\alpha^{-a}(g)}$. Since
$\bar{\alpha}\widehat{g}=\widehat{\alpha(g)}\bar{\alpha}$
and
$\bar{\alpha},\gamma$ commute, we get
$\bar{\alpha}g'=(\alpha(g))'\bar{\alpha}$. Also observe that
as $G$ is abelian, $g'$ commutes with $\widehat{h}$ for every $g,h\in G$.
Hence

\begin{eqnarray*}
g'k&=&\gamma^{-1}\widehat{g}\gamma\tau^a\widehat{l}=\gamma^{-1}\widehat{g}\tau^a\gamma\bar{\alpha}^a\widehat{l}=\gamma^{-1}\tau^a\widehat{g}\bar{\alpha}^a\gamma\widehat{l}\\
&=&\tau^a\gamma^{-1}\bar{\alpha}^{-a}\widehat{g}\bar{\alpha}^a\gamma\widehat{l}=\tau^a(\alpha^{-a}(g))'\widehat{l}=\tau^a
\widehat{l}(\alpha^{-a}(g))'\\
&=&k\widehat{\alpha^{-a}(g)}\widehat{\alpha^{-a}(g)}^{-1}\gamma^{-1}\widehat{\alpha^{-a}(g)}\gamma=m[\widehat{\alpha^{-a}(g)},\gamma]    \end{eqnarray*}
and
\begin{eqnarray*}
g'[k((0,0),s)]&=&g'k((0,0),[\widehat{x},\gamma](1,0),[\widehat{x},\gamma](2,0))\\
&=&m[\widehat{\alpha^{-a}(g)},\gamma]((0,0),[\widehat{x},\gamma](1,0),[\widehat{x},\gamma](2,0))\\
&=&m((0,0),[\widehat{\alpha^{-a}(g)},\gamma][\widehat{x},\gamma](1,0),[\widehat{\alpha^{-a}(g)},\gamma][\widehat{x},\gamma](2,0))\\
&=&m((0,0),[\widehat{\alpha^{-a}(g)x},\gamma](1,0),[\widehat{\alpha^{-a}(g)x},\gamma](2,0))\in E(X).
\end{eqnarray*}

\noindent This proves that $g'\in \Aut(X)$. Since $g$ is an arbitrary element of
$G$, we have $\gamma^{-1}G_L\gamma\subseteq \Aut(X)$. As claimed,
$\gamma^{-1}(\mathbb{Z}_p\times G)_L\gamma$ is a regular
subgroup of $\Aut(X)$ conjugate in $\Sym(\mathbb{Z}_p\times G)$ to
$(\mathbb{Z}_p\times G)_L$.

We now have that $Y = \gamma(X)$ is a Cayley (color)
ternary relational structure of $\Z_p\times G$ as $\Aut(Y) =
\gamma\Aut(X)\gamma^{-1}$.  We will next show that $Y\not = X$.  Assume by way of contradiction that $Y=X$. As $\gamma(0,g)=(0,g)$ for every
$g\in G$, the permutation $\gamma$ must map edges of
$U$ to themselves, so that $\gamma(S-U) = S - U$.  We will show that $\gamma(S-U)\not = S - U$. Note that we need not concern ourselves with colors because as all the edges derived from $S - U$ via translations of $(\Z_p\times G)_L$ have the same color.  Observing that

\begin{eqnarray*}
([\hat{g},\gamma](1,0),[\hat{g},\gamma](2,0)) & = &
(\hat{g}^{-1}\gamma^{-1}\hat{g}\gamma(1,0),\hat{g}^{-1}\gamma^{-1}\hat{g}\gamma(2,0)) =
(\hat{g}^{-1}\gamma^{-1}\hat{g}(1,0),\hat{g}^{-1}\gamma^{-1}\hat{g}(2,0))\\
& = & (\hat{g}^{-1}\gamma^{-1}(1,g),\hat{g}^{-1}\gamma^{-1}(2,g))
 =
(\hat{g}^{-1}(1,\alpha^{-1}(g),\hat{g}^{-1}(2,\alpha^{-2}(g))\\
& = & ((1,\alpha^{-1}(g) - g),(2,\alpha^{-2}(g) - g)),
\end{eqnarray*}

\noindent we see that $\gamma(S - U) = \{((1,g - \alpha(g)),(2,g -
\alpha^2(g))):g\in G\}$. Moreover, as $S-U=\{(1,\alpha^{-1}(g)-g),(2,\alpha^{-2}(g)-g): g \in G\}$, we conclude that for each $g\in G$,
there exists $h_g\in G$ such that

$$g - \alpha(g) = \alpha^{-1}(h_g) - h_g{\rm\ \ \ \ \ and\ \ \ \ \ }g
- \alpha^2(g) = \alpha^{-2}(h_g) - h_g.$$

\noindent Setting $\iota:G\to G$ to be the identity permutation, we
may rewrite the above equations as

$$(\iota - \alpha)(g) = (\alpha^{-1} - \iota)(h_g){\rm \ \ \ \ \ and \ \
\ \ \ }(\iota - \alpha^2)(g) = (\alpha^{-2} - \iota)(h_g).$$

\noindent Computing in the endomorphism ring of the abelian group $G$, we see that $(\alpha^{-2} - \iota) =
(\alpha^{-1}+\iota)(\alpha^{-1} - \iota)$.  Applying the
endomorphism $(\alpha^{-1} + \iota)$ to the first equation above, we
then have that

$$
(\alpha^{-1} + \iota)(\iota - \alpha)(g)  =  (\alpha^{-1} +
\iota)(\alpha^{-1} - \iota)(h_g)
 =  (\alpha^{-2} - \iota)(h_g)
 =  (\iota - \alpha^2)(g).$$

\noindent Hence $(\alpha^{-1} + \iota)(\iota - \alpha) = \iota -
\alpha^2$, and so $$0=(\alpha^{-1}+\iota)(\iota-\alpha)-(\iota-\alpha^2)=((\alpha^{-1}+\iota)-(\iota+\alpha))(\iota-\alpha)=(\alpha^{-1}-\alpha)(\iota-\alpha),$$
(here
$0$ is the endomorphism of $G$ that maps each element of $G$ to
$0$).
As $\alpha$ fixes only $0$, the endomorphism $\iota-\alpha$ is invertible, and so
we see that $\alpha^{-1} -\alpha= 0$, and $\alpha = \alpha^{-1}$. However,
this implies that $p=\vert\alpha\vert = 2$, a contradiction. Thus
$\gamma(S - U)\not = S - U$ and so $Y\not = X$.

We set $T = \gamma(S)$, so that $((0,0),t)\in E(Y)$ for every $t\in T$, where if $X$ is a color ternary relational structure we assume that $\gamma$ preserves colors.  Now suppose that there exists $\beta\in\Aut(\Z_p\times G)$ such
that $\beta(X) = Y$. Since $\gcd(p,\vert G\vert)=1$, we obtain that $\Z_p\times 1_G$ and $1_{\Z_p}\times G$ are characteristic
subgroups of $\Z_p\times G$. Therefore $\beta(i,j) = (\beta_1(i),\beta_2(j))$,
where $\beta_1\in\Aut(\Z_p)$ and $\beta_2\in\Aut(G)$.  As
$\beta$ fixes $(0,0)$, we must have that $\beta(S) = T$. As there
is no element of $T$ of the form $((2,x_1),(1,y_1))$, we conclude
that $\beta_1 = 1$ as $\beta_1(i) = i$ or $2i$.  As
$\bar{\alpha}\in\Aut(X)$ and $X\neq Y$, we have that
$\beta_2\not\in\la\alpha\ra$.  Now observe that $\beta(U) = U$.
Thus $\beta_2\in\Aut(Z) = \la G_L,\alpha\ra$. We conclude that
$\beta_2\in\la\alpha\ra$, a contradiction. Thus $X,Y$ are not
isomorphic by a group automorphism of $\Z_p\times G$, and the result follows.
\end{proof}

The following two lemmas, which in our opinion are of independent interest, will be used (together with Theorem~\ref{main}) in the proof of Corollary~\ref{coro1}.

\begin{lem}\label{Stabkary}
Let $G$ be a transitive permutation group on $\Omega$.  If $x\in\Omega$ and $\Stab_G(x)$ in its action on $\Omega - \{x\}$ is the automorphism group of a $k$-ary relational structure with vertex set $\Omega - \{x\}$, then $G$ is the automorphism group of a $(k+1)$-ary relational structure.
\end{lem}

\begin{proof}
Let $Y$ be a $k$-ary relational structure with vertex set $\Omega - \{x\}$ and automorphism group $\Stab_G(x)$ in its action on $\Omega - \{x\}$.  Let $W = \{(x,v_1,\ldots,v_k):(v_1,\ldots,v_k)\in E(Y)\}$, and define a $(k+1)$-ary relational structure $X$ by $V(X) = \Omega$ and $E(X) = \{g(w):w\in W \,\mathrm{ and }\, g\in G \}$.  We claim that $\Aut(X) = G$.  First, observe that $\Stab_G(x)$ maps $W$ to $W$. Also, if $e\in E(X)$ and $e = (x,v_1,\ldots,v_k)$ for some $v_1,\ldots,v_k\in \Omega$, then there exists $ (x,u_1,\ldots,u_k)\in W$ and $g\in G$ with $g(x,u_1,\ldots,u_k) = (x,v_1,\ldots,v_k)$.  We conclude that $g(x) = x$ and $g(u_1,\ldots,u_k) = (v_1,\ldots,v_k)$.  Hence $g\in\Stab_G(x)$ and $(v_1,\ldots,v_k)\in E(Y)$.  Then $W$ is the set of all edges of $X$ with first coordinate $x$.

By construction, $G\le\Aut(X)$.  For the reverse inclusion, let $h\in\Aut(X)$.  As $G$ is transitive, there exists $g\in G$ such that $g^{-1}h\in\Stab_{\Aut(X)}(x)$.  Note that as $g\in G$, the element $g^{-1}h\in G$ if and only if $h\in G$.  We may thus assume without loss of generality that $h(x) = x$.  Then $h$ stabilizes set-wise the set of all edges of $X$ with first coordinate $x$, and so $h(W) = W$ and $h$ induces an automorphism of $Y$. As $\Aut(Y) = \Stab_G(x)\leq G$, the result follows.
\end{proof}

\begin{lem}\label{autosemiregular}
Let $m\ge 2$ be an integer and $\rho\in \Sym(\Z_{ms})$ be a semiregular element of order $m$ with $s$ orbits.  Then there exists a digraph with vertex set $\Z_{ms}$ and with automorphism group $\la\rho\ra$.
\end{lem}

\begin{proof}
For each $i\in \Z_s$, set $$\rho_i = (0,1,\ldots,m-1)\cdots(im,im + 1,\ldots,im+m-1)\quad \textrm{and}\quad V_i = \{im + j:j\in\Z_m\}.$$   We inductively define a sequence of graphs $\Gamma_0,\ldots,\Gamma_{s-1} = \Gamma$ such that the subgraph of $\Gamma$ induced by $\Z_{(i+1)m}$ is $\Gamma_i$, the indegree in $\Gamma$ of a vertex in $V_i$ is $i+1$, and $\Aut(\Gamma_i) = \la\rho_i\ra$, for each $i\in\Z_s$.

We set $\Gamma_0$ to be the directed cycle of length $m$ with edges $\{(j,j+1): j \in \Z_m\}$ and with automorphism group $\la\rho_0\ra$.  Inductively assume that $\Gamma_{s-2}$, with the above properties, has been constructed.  We construct $\Gamma_{s-1}$ as follows.  First, the subgraph of $\Gamma_{s-1}$ induced by $\Z_{(s-1)m}$  is $\Gamma_{s-2}$.  Then we place the directed $m$ cycle $\{((s-1)m+j,(s-1)m+j+1):j \in \Z_m\}$  whose automorphism group is $\la((s-1)m,(s-1)m + 1,\ldots,(s-1)m+m-1)\ra$  on the vertices in $V_{s-1}$.  Additionally, we declare the vertex $(s-1)m$  to be outadjacent to $(s-2)m$ and to every vertex that $(s-2)m$  is outadjacent to that is not contained in $V_{s-2}$. Finally, we add to $\Gamma_{s-1}$ every image of one of these edges under an element of $\langle\rho_{s-1}\rangle$.

By construction, $\rho_{s-1}$ is an automorphism of $\Gamma_{s-1}$ and the subgraph of $\Gamma_{s-1}$ induced by $\Z_{(s-1)m}$  is $\Gamma_{s-2}$.  Then each vertex in $\Gamma_{s-1}\cap V_i$ has indegree $i+1$ for $0\le i\le s-2$, while it is easy to see that each vertex of $V_{s-1}$ has indegree $s$.  Finally, if $\delta\in \Aut(\Gamma_{s-1})$, then $\delta$ maps vertices of indegree $i+1$ to vertices of indegree $i+1$, and so $\delta$ fixes set-wise $V_i$, for every $i\in\Z_s$.  Additionally, the action induced by $\langle\delta\rangle$ on $V_{s-1}$ is necessarily $\la((s-1)m,(s-1)m + 1,\ldots,(s-1)m+m-1)\ra$  as this is the automorphism group of the subgraph of $\Gamma_{s-1}$ induced by $V_{s-1}$. Moreover, arguing by induction, we may assume that the action induced by $\delta$ on $V(\Gamma_{s-1}) - V_{s-1}$ is given by an element of $\langle\rho_{s-2}\rangle$.  If $\delta\not\in\la\rho_{s-1}\ra$, then $\Aut(\Gamma_{s-1})$ has order at least $m^2$, and there is some element of $\Aut(\Gamma_{s-1})$ that is the identity on $V(\Gamma_{s-2})$ but not on $V_{s-1}$ and vice versa.  This however is not possible as each vertex of $V_{s-2}$ is outadjacent to exactly one vertex of $V_{s-1}$.  Then $\Aut(\Gamma_{s-1}) = \la\rho_{s-1}\ra$ and the result follows.
\end{proof}

\begin{cor}\label{coro1}
None of the groups $\Z_3\times\Z_2^2$, $\Z_7\times\Z_2^3$, or $\Z_5\times\Z_2^4$ are CI-groups with
respect to ternary relational structures.
\end{cor}

\begin{proof}
Observe that $\Z_2^2$ has an automorphism $\alpha_3$ of order $3$ that fixes $0$ and acts regularly on the remaining $3$ elements, and similarly, $\Z_2^3$ has an automorphism $\alpha_7$ of order $7$ that fixes $0$ and acts regularly on the remaining $7$ elements.  As a regular cyclic group is the automorphism group of a directed cycle, we see that $\la(\Z_3\times\Z_2^2)_L,\alpha_3\ra$ and $\la(\Z_7\times\Z_2^3)_L,\alpha_7\ra$ are the automorphism groups of ternary relational structures by Lemma \ref{Stabkary}.  The result then follows by Theorem \ref{main}.

Now $\Z_2^4$ has an automorphism $\alpha_5$ of order $5$ that fixes $0$ and acts semiregularly on the remaining $15$ points.  Then $\la\alpha_5\ra$ in its action on $\Z_2^4 - \{0\}$ is the automorphism group of a binary relational structure by Lemma \ref{autosemiregular}. By Lemma~\ref{Stabkary}, there exists a ternary relational structure with automorphism group $\la(\Z_5\times\Z_2^4)_L,\alpha_5\ra$.  The result then follows by Theorem \ref{main}.
\end{proof}

Before proceeding, we will need terms and notation concerning complete block systems.

Let $G\le \Sym(n)$ be a transitive permutation group (acting on $\Z_n$, say).  A subset $B\subseteq\Z_n$ is a {\it block for $G$} if $g(B) = B$ or $g(B)\cap B = \emptyset$ for every $g\in G$.  Clearly $\Z_n$ and its singleton subsets are always blocks for $G$, and are called {\it trivial blocks}.  If $B$ is a block, then $g(B)$ is a block for every $g\in G$, and the set ${\cal B} = \{g(B):g\in G\}$ is called a {\it complete block system for $G$}, and we say that $G$ {\it admits} ${\cal B}$.  A complete block system is {\it nontrivial} if its blocks are nontrivial.  Observe that a complete block system is a partition of $\Z_n$, and any two blocks have the same size.  If $G$ admits ${\cal B}$ as a complete block system, then each $g\in G$ induces a permutation of ${\cal B}$, which we denote by $g/{\cal B}$.  We set $G/{\cal B} = \{g/{\cal B}:g\in G\}$.  The kernel of the action of $G$ on ${\cal B}$, denoted by $\fix_G({\cal B})$, is then the subgroup of $G$ which fixes each block of ${\cal B}$ set-wise.  That is, $\fix_G({\cal B}) = \{g\in G:g(B) = B{\rm\ for\ all\ }B\in{\cal B}\}$.  For fixed $B\in{\cal B}$, we denote the set-wise stabilizer of $B$ in $G$ by $\Stab_G(B)$.  That is $\Stab_G({\cal B}) = \{g\in G:g(B) = B\}$.  Note that $\fix_G({\cal B}) = \cap_{B\in{\cal B}}\Stab_G(B)$.  Finally, for $g\in\Stab_G(B)$, we denote by $g\vert_B$ the action induced by $g$ on $B\in{\cal B}$.

Note that Corollary~\ref{coro1}, together with the fact that $\Z_2^3\times\Z_p$, $p\ge 11$, is a CI-group with respect to color ternary relational structures~\cite{Dobson2010a}, settles the question of which groups $\Z_2^3\times\Z_p$ are CI-groups with respect to color ternary relational structures  except for $p = 5$.   Our next goal is to show that $\Z_2^3\times\Z_5$ is a CI-group with respect to color ternary relational structures. From a computational point of view, the number of points is too large to enable a computer to determine the answer without some additional information.
Lemma~$6.1$ in~\cite{Dobson2010a} is the only result that uses the hypothesis $p\geq 11$.  For convenience, we report~\cite[Lemma~$6.1$]{Dobson2010a}.
\begin{lem}\label{refine}
Let $p\geq 11$ be a prime and write $H=\Z_2^3\times \Z_p$. For every $\phi\in \Sym(H)$, there exists $\delta\in \la H_L,\phi^{-1}H_L\phi\ra$ such that $\la H_L,\delta^{-1}\phi^{-1}H_L\phi\delta\ra$ admits a complete block system consisting of $8$ blocks of size $p$.
\end{lem}

In particular, to prove that $\Z_2^3\times \Z_5$ is a CI-group with respect to color ternary relational structures, it suffices to prove that Lemma~\ref{refine} holds true also for the prime $p=5$. We begin with some intermediate results which accidentally will also help us to prove that $\Z_2^3\times\Z_7$ is a CI-group with respect to color binary relational structures. (Here we denote by $\Alt(X)$ the alternating group on the set $X$ and by $\Alt(n)$ the alternating group on $\{1,\ldots,n\}$.)

\begin{lem}\label{dtlemma}
Let $P_1$ and $P_2$ be partitions of $\Z_n$ where each block in $P_1$ and $P_2$ has order $p\geq 2$.  Then there exists $\phi\in \Alt(\Z_n)$ such that $\phi(P_1) = P_2$.
\end{lem}

\begin{proof}
Let $P_1 = \{\Delta_1,\ldots,\Delta_{n/p}\}$ and $P_2 = \{\Omega_1,\ldots,\Omega_{n/p}\}$. As $\Alt(n)$ is $(n-2)$-transitive, there exists $\phi\in \Alt(n)$ such that $\phi(\Delta_i) = \Omega_i$, for $i\in\{1,\ldots, n/p - 1\}$.  As both $P_1$ and $P_2$ are partitions, we see that $\phi(\Delta_{n/p}) = \Omega_{n/p}$ as well.
\end{proof}

\begin{lem}\label{altlem}
Let $n = 8p$, $G = (\Z_2^3\times\Z_p)_L$ and $\delta\in \Sym(n)$. Suppose that $\la G,\delta^{-1}G\delta\ra$ admits a complete block system ${\cal C}$ with $p$ blocks of size $8$ such that $\Alt(C)\le\Stab_{\la G,\delta^{-1}G\delta\ra}(C)\vert_C$, where $C\in{\cal C}$.  Then there exists $\gamma\in\la G,\delta^{-1}G\delta\ra$ such that $\la G,\gamma^{-1}\delta^{-1}G\delta\gamma\ra$ admits a complete block system ${\cal B}$ with $4p$ blocks of size $2$.
\end{lem}

\begin{proof}
Clearly both $G$ and $\delta^{-1}G\delta$ are regular, and so both $\fix_G({\cal C})$ and $\fix_{\delta^{-1}G\delta}({\cal C})$ are semiregular of order $8$.  As $\Alt(8)$ is simple and as $\fix_{\la G,\delta^{-1}G\delta\ra}({\cal C})\vert_C\tl\Stab_{\la G,\delta^{-1}G\delta\ra}({\cal C})\vert_C$, we have that $\Alt(C)\le\fix_{\la G,\delta^{-1}G\delta\ra}({\cal C})\vert_C$, for every $C\in{\cal C}$.  Let $J\le\fix_G({\cal C})$ and
$K\le\fix_{\delta^{-1}G\delta}({\cal C})$ be both of order $2$. Fix $C_0\in{\cal C}$, and let ${\cal O}_1,\ldots,{\cal O}_4$ be the orbits of $J\vert_{C_0}$, and ${\cal O}_1',\ldots,{\cal O}_4'$ be the orbits of $K\vert_{C_0}$. By Lemma~\ref{dtlemma}, there exists $\gamma_0\in\fix_{\la G,\delta^{-1}G\delta\ra}({\cal C})$ such
that $\gamma_0^{-1}({\cal O}_i') = {\cal O}_i$, for each $i\in\{1,\ldots,4\}$.  Hence the orbits of $J\vert_{C_0}$ and $(\gamma_0^{-1}K\gamma_0)\vert_{C_0}$ are identical.

Recall that two  transitive actions are equivalent if and only if the stabilizer of a point in one action is the same as the stabilizer of a point in the other \cite[Lemma~1.6B]{DixonM1996}. Suppose now that the action  of $\fix_{\la G,\delta^{-1}G\delta\ra}({\cal
C})$ on $C_0$ is equivalent to the action of $\fix_{\la G,\delta^{-1}G\delta\ra}({\cal C})$ on $C\in{\cal C}$.  Let $\omega_J$ generate $J$ and let $\omega_K$ generate $K$.  As the orbits of $J\vert_{C_0}$ and $(\gamma_0^{-1}K\gamma_0)\vert_{C_0}$ are identical and $|\omega_J|=|\omega_K|=2$, we see that $\omega_J\vert_{C_0} = (\gamma_0^{-1}\omega_K\gamma_0)\vert_{C_0}$.  Hence $(\omega_J\gamma_0^{-1}\omega_K\gamma_0)\vert_{C_0} = 1$ and so $(\omega_J\gamma_0^{-1}\omega_K\gamma_0)\vert_{C} = 1$. Therefore the orbits of $J\vert_{C}$ and $(\gamma_0^{-1}K\gamma_0)\vert_{C}$ are identical.

Define an equivalence relation $\equiv$ on ${\cal C}$ by $C\equiv C'$ if and only if the action of $\fix_{\la G,\delta^{-1}G\delta\ra}({\cal C})$ on $C$ is equivalent to the action of $\fix_{\la G,\delta^{-1}G\delta\ra}({\cal C})$ on $C'$.  Since $\Alt(8)$ has only one permutation representation of degree $8$~\cite[Theorem~5.3]{Cameron1981}, we obtain that $C\not\equiv C'$ if and only if the action of $\fix_{\la G,\delta^{-1}G\delta\ra}({\cal C})\vert_{C\cup C'}$ on $C'$ is not faithful.  Thus $C\not \equiv C'$ if and only if there exists $\alpha\in\fix_{\la G,\delta^{-1}G\delta\ra}({\cal C})$ such that $\alpha\vert_C =
1$ but $\alpha\vert_{C'}\not = 1$.

Let $E_0$ be the $\equiv$-equivalence class containing $C_0$ and set $$L_1 = \{\alpha\in\fix_{\la G,\delta^{-1}G\delta\ra}({\cal C}):\alpha\vert_{C} = 1 \textrm{ for every }C\in E_0\}.$$  Let  $C_1$ be in ${\cal C}$ with $C_1\not\equiv C_0$ and let $E_1$ be the $\equiv$-equivalence class containing $C_1$. Then there exists $\omega\in\fix_{\la G,\delta^{-1}G\delta\ra}({\cal C})$ with $\omega\vert_{C_0} = 1$ and $\omega\vert_{C_1}\not = 1$. From the definition of $\equiv$, we see that $\omega\vert_C = 1$, for every $C\in E_0$, that is, $\omega\in L_1$  and $L_1\not = 1$. As $L_1\tl\fix_{\la G,\delta^{-1}G\delta\ra}({\cal C})$ and $\Alt(8)$ is simple, we conclude that $\Alt(C_1)\le L_1\vert_{C_1}$.

As both $J$ and $K$ are semiregular of order $2$, the groups $J\vert_{C_1}$ and $(\gamma_0^{-1}K\gamma_0)\vert_{C_1}$ are generated by even permutations. So $J\vert_{C_1}\le L_1\vert_{C_1}$ and $(\gamma_0^{-1} K\gamma_0)\vert_{C_1}\le L_1\vert_{C_1}$.  By Lemma~\ref{dtlemma}, there exists $\gamma_1\in L_1$ such that the orbits of $J\vert_{C_1}$ and $(\gamma_1^{-1}\gamma_0^{-1}K\gamma_0\gamma_1)\vert_{C_1}$ are identical. In particular, the orbits of $J\vert_{C}$ and $(\gamma_1^{-1}\gamma_0^{-1}K\gamma_0\gamma_1)\vert_{C}$ are identical, for every $C\in E_1$.  Furthermore, as $L_1\vert_C = 1$ for every $C\in E_0$, we have that the orbits $J\vert_C$ and $(\gamma_1^{-1}\gamma_0^{-1}K\gamma_0\gamma_1)\vert_C$ are identical for every $C\in E_0\cup E_1$.

Applying inductively the previous two paragraphs to the various $\equiv$-equivalence classes, we find $\gamma\in\la G,\delta^{-1}G\delta\ra$ such that the orbits of $J$ and $(\gamma^{-1}\delta^{-1}K\delta\gamma)$ are identical. Since $|J|=2$, we get $J=\gamma^{-1}\delta^{-1}K\delta\gamma$. As $J\tl G$ and $\gamma^{-1}\delta^{-1}K\delta\gamma\tl\gamma^{-1}\delta^{-1}G\delta\gamma$, we obtain $J\tl\la G,\gamma^{-1}\delta^{-1}G\delta\gamma\ra$ and the orbits of $J$ form a complete block system for $\la G,\gamma^{-1}\delta^{-1}G\delta\gamma\ra$  of $4p$ blocks of size $2$.
\end{proof}

The proof of the following result is analogous to the proof of \cite[Lemma 6.1]{Dobson2010a}.

\begin{lem}\label{pbig}
Let $H$ be an abelian group of order $\ell p$, where $\ell < p$ and $p$ is prime.  Let $\phi\in \Sym(H)$.  Then there exists $\delta\in \la H_L,\phi^{-1}H_L\phi\ra$ such that $\la H_L,\delta^{-1}\phi^{-1}H_L\phi\delta\ra$ admits a complete block system with blocks of size $p$.
\end{lem}

\begin{lem}\label{57tool}
Let $p \ge 5$, $H = \Z_2^3\times\Z_p$,  and $\phi\in \Sym(H)$. Then either there exists $\delta\in \la H_L,\phi^{-1}H_L\phi\ra$ such that $\la H_L,\delta^{-1}\phi^{-1}H_L\phi\delta\ra$ admits a complete block system with blocks of size $p$ or $ \la H_L,\phi^{-1}H_L\phi\ra$ admits a complete block system ${\cal B}$ with blocks of size $8$ and $\fix_K({\cal B})\vert_B$ is isomorphic to a primitive subgroup of $\AGL(3,2)$, for $B\in {\cal B}$.
\end{lem}

\begin{proof}
Set $K=\la H_L,\phi^{-1}H_L\phi\ra$. As $H$ has a cyclic Sylow $p$-subgroup, we have by~\cite[Theorem 3.5A]{DixonM1996} that $K$ is doubly-transitive or imprimitive.  If $K$ is doubly-transitive, then by~\cite[Theorem 1.1]{Li2003} we have that $\Alt(H)\le K$. Now Lemma~\ref{dtlemma} reduces this case to the imprimitive case.  Thus we may assume that $K$ is imprimitive with a complete block system ${\cal C}$.

Suppose that the blocks of ${\cal C}$ have size $\ell p$, where $\ell = 2$ or $4$.  Notice  that $p > \ell$.  As $H$ is abelian, $\fix_{H_L}({\cal C})$ is a semiregular group of order $\ell p$ and $\fix_{\phi^{-1}H_L\phi}({\cal C})$ is also a semiregular group of order $\ell p$.  Then, for $C\in{\cal C}$, both $\fix_{H_L}({\cal C})\vert_C$ and $\fix_{\phi^{-1}H_L\phi}({\cal C})\vert_C$ are regular groups of order $\ell p$.  Let $C\in{\cal C}$.  By Lemma \ref{pbig}, there exists $\delta\in \la\fix_{H_L}({\cal C}),\fix_{\phi^{-1}H_L\phi}({\cal C})\ra$ such that $\la\fix_{H_L}({\cal C}),\fix_{\delta^{-1}\phi^{-1}H_L\phi\delta}({\cal C})\ra\vert_{C}$ admits a complete block system ${\cal B}_{C}$ consisting of blocks of size $p$.  Let $C'\in{\cal C}$ with $C'\not = C$.  Arguing as above, there exists $\delta'\in \la\fix_{H_L}({\cal C}),\fix_{\delta^{-1}\phi^{-1}H_L\phi\delta}({\cal C})\ra$ such that $\la\fix_{H_L}({\cal C}),\fix_{\delta'^{-1}\delta^{-1}\phi^{-1}H_L\phi\delta\delta'}({\cal C})\ra\vert_{C'}$ admits a complete block system ${\cal B}_{C'}$ consisting of blocks of size $p$.  Note that $\delta'\vert_{C}\in\la\fix_{H_L}({\cal C}),\fix_{\delta^{-1}\phi^{-1}H_L\phi\delta}({\cal C})\ra\vert_{C}$ and so  $\la\fix_{H_L}({\cal C}),\fix_{\delta'^{-1}\delta^{-1}\phi^{-1}H_L\phi\delta\delta'}({\cal C})\ra\vert_{C}$ admits ${\cal B}_{C'}$ as a complete block system.  Repeating this argument for every block in ${\cal C}$, we find $\delta\in \la\fix_{H_L}({\cal C}),\fix_{\phi^{-1}H_L\phi}({\cal C})\ra$ such that $\la\fix_{H_L}({\cal C}),\fix_{\delta^{-1}\phi^{-1}H_L\phi\delta}({\cal C})\ra\vert_{C}$ admits a complete block system ${\cal B}_C$ consisting of blocks of size $p$.  Let ${\cal B} = \cup_{C}{\cal B}_C$.  We claim that ${\cal B}$ is a complete block system for $\la H_L,\delta^{-1}\phi^{-1}H_L\phi\delta\ra$, which will complete the argument in this case.

Let $\rho\in H_L$ be of order $p$.  By construction, $\rho\in\fix_{H_L}({\cal B})$. As $H$ is abelian, $\fix_{H_L}({\cal C})\vert_{C}$ is abelian, for every $C\in{\cal C}$.  Then ${\cal B}_C$ is formed by the orbits of some subgroup of $\fix_{H_L}({\cal C})\vert_{C}$ of order $p$, and as $\la\rho\ra\vert_{C}$ is the unique subgroup of $\fix_{H_L}({\cal C})\vert_{C}$ of order $p$, we obtain that ${\cal B}_C$ is formed by the orbits of $\la\rho\ra\vert_{C}$.  Then ${\cal B}$ is formed by the orbits of $\la\rho\ra\tl H_L$ and ${\cal B}$ is a complete block system for $H_L$.  An analogous argument for $\delta^{-1}\phi^{-1}\la\rho\ra\phi\delta$ gives that ${\cal B}$ is  a complete block system for $\delta^{-1}\phi^{-1}H_L\phi\delta$.  Then ${\cal B}$ is a complete block system for $\la H_L,\delta^{-1}\phi^{-1}H_L\phi\delta\ra$ with blocks of size $p$, as required.

Suppose that the blocks of ${\cal C}$ have size $8$.  Now $H_L/{\cal C}$ and $\phi^{-1}H_L\phi/{\cal C}$ are cyclic  of order $p$, and as $\Z_p$ is a CI-group \cite[Theorem 2.3]{Babai1977}, replacing $\phi^{-1}H_L\phi$ by a suitable conjugate, we may assume that $\la H_L,\phi^{-1}H_L\phi\ra/{\cal C} = H_L/{\cal C}$.  Then $K/{\cal C}$ is regular and $\Stab_K(C) = \fix_K({\cal C})$, for every $C\in{\cal C}$.

Suppose that $\Stab_K({\cal C})\vert_C$ is imprimitive, for $C\in{\cal C}$.  By \cite[Exercise 1.5.10]{DixonM1996}, the group $K$ admits a complete block system ${\cal D}$ with blocks of size $2$ or $4$.  Then $K/{\cal D}$ has degree $2p$ or $4p$ and, by Lemma~\ref{pbig}, there exists $\delta\in K$ such that $\la H_L,\delta^{-1}\phi^{-1}H_L\phi\delta\ra/{\cal D}$ admits a complete block system ${\cal B}'$ with blocks of size $p$.  In particular, ${\cal B}'$ induces a complete block system ${\cal B}''$ for $\la H_L,\delta^{-1}\phi^{-1}H_L\phi\delta\ra$ with blocks of size $2p$ or $4p$, and we conclude by the case previously considered applied with $\mathcal{C}={\cal B}''$.  Suppose that $\Stab_K({\cal C})\vert_C$ is primitive, for $C\in{\cal C}$. If $\Stab_K({\cal C})\vert_C\ge \Alt(C)$, then the result follows by Lemma~\ref{altlem}, and so we may assume this is not the case.  By~\cite[Theorem 1.1]{Li2003}, we see that $\Stab_K({\cal C})\vert_C\le\AGL(3,2)$.  The result now follows with ${\cal B} = {\cal C}$.
\end{proof}

\begin{cor}\label{5tool}
Let $H = \Z_2^3\times\Z_5$ and $\phi\in \Sym(H)$.  Then there exists $\delta\in \la H_L,\phi^{-1}H_L\phi\ra$ such that $\la H_L,\delta^{-1}\phi^{-1}H_L\phi\delta\ra$ admits a complete block system with blocks of size $5$.
\end{cor}

\begin{proof}
Set $K=\la H_L,\phi^{-1}H_L\phi\ra$.
By Lemma~\ref{57tool}, we may assume  that $K$ admits a complete block system ${\cal B}$ with blocks of size $8$ and with $\Stab_K({\cal B})\vert_B\le\AGL(3,2)$, for $B\in {\cal B}$.  As $\vert\AGL(3,2)\vert = 8\cdot 7\cdot 6\cdot 4$, we see that a Sylow $5$-subgroup of $K$ has order $5$.  Let $\la\rho\ra$ be the  subgroup of $H_L$ of order $5$. So $\la\rho\ra$ is a Sylow $5$-subgroup of $K$.  Then $\phi^{-1}\la\rho\ra\phi$ is also a Sylow $5$-subgroup of $K$, and by a Sylow theorem there exists $\delta\in K$ such that $\delta^{-1}\phi^{-1}\la\rho\ra\phi\delta = \la\rho\ra$.  We then have that $\la H_L,\delta^{-1}\phi^{-1}H_L\phi\delta\ra$ has a unique Sylow $5$-subgroup, whose orbits form the required complete block system ${\cal B}$.
\end{proof}

We are finally ready to prove Theorem~A.
\begin{proof}[Proof of Theorem~A]
If $p$ is odd, then the paragraph following the proof of Corollary~\ref{coro1} shows that it suffices to prove that Lemma~\ref{refine} holds for the prime $p=5$. This is done in Corollary~\ref{5tool}. If $p=2$, then the result can be verified using GAP or Magma.
\end{proof}

\section{Proof of Corollaries~A and~B}
Before proceeding to our next result we will need the following definitions.

\begin{defin}\label{defin1}{\rm
Let $G$ be a permutation group on $\Omega$ and $k\geq 1$. A permutation $\sigma\in \Sym(\Omega)$ lies in the $k$-closure  $G^{(k)}$ of $G$
if for every $k$-tuple $t\in\Omega^k$ there exists $g_t\in G$
(depending on $t$) such that $\sigma(t) = g_t(t)$. We say that $G$ is
$k$-closed if the permutations lying in the $k$-closure of $G$
are the elements of $G$, that is, $G^{(k)}=G$. The group $G$ is $k$-closed if and only if there exists a color $k$-ary relational structure $X$ on $\Omega$ with  $G=\Aut(X)$, see~\cite{Wielandt1969}.}
\end{defin}

\begin{defin}
{\rm For color digraphs $\Gamma_1$ and $\Gamma_2$, we define the {\it wreath product of $\Gamma_1$ and $\Gamma_2$}, denoted $\Gamma_1\wr\Gamma_2$, to be the color digraph with vertex set $V(\Gamma_1)\times V(\Gamma_2)$ and edge set $E_1\cup E_2$, where  $E_1 =\{((x_1,y_1),(x_1,y_2)):x_1\in V(\Gamma_1), (y_1,y_2)\in E(\Gamma_2)\}$ and the edge $((x_1,y_1),(x_1,y_2))\in E_1$ is colored with the same color as $(y_1,y_2)$ in $\Gamma_2$, and $E_2 = \{((x_1,y_1),(x_2,y_2)):(x_1,x_2)\in E(\Gamma_1), y_1,y_2\in V(\Gamma_2)\}$ and the edge $((x_1,y_1),(x_2,y_2))
\in E_2$ is colored with the same color as $(x_1,x_2)$ in $\Gamma_1$.  }
\end{defin}

\begin{defin}
{\rm For permutation groups $G\le \Sym(X)$ and $H\le \Sym(Y)$, we define the {\it wreath product of $G$ and $H$}, denoted $G\wr H$, to be the permutation group $G\wr H\le \Sym(X\times Y)$ consisting of all permutations of the form $(x,y)\mapsto(g(x),h_x(y))$, $g\in G$, $h_x\in H$.}
\end{defin}

The following very useful result (see \cite[Lemma 3.1]{Babai1977}) characterizes CI-groups with respect to a class of combinatorial objects.

\begin{lem}\label{CItool}
Let $H$ be a group and let ${\cal K}$ be a class of combinatorial objects.  The following are equivalent.
\begin{enumerate}
\item $H$ is a CI-group with respect to ${\cal K}$,
\item whenever $X$ is a Cayley object of $H$ in ${\cal K}$ and $\phi\in \Sym(H)$ such that $\phi^{-1}H_L\phi\le\Aut(X)$, then $H_L$ and $\phi^{-1}H_L\phi$ are conjugate in $\Aut(X)$.
\end{enumerate}
\end{lem}

\begin{proof}[Proof of Corollary~A]
From Theorem~A, it suffices to show that $\Z_2^3\times\Z_3$ and $\Z_2^3\times\Z_7$ are CI-groups with respect to color binary relational structures.
As the transitive permutation groups of degree $24$ are readily available in GAP or Magma, it can be shown using a computer that $\Z_2^3\times \Z_3$ is a CI-group with respect to color binary relational structures. It remains to consider $H=\Z_2^3\times \Z_7$.

Fix $\phi\in \Sym(H)$ and set $K=\la H_L,\phi^{-1}H_L\phi\ra$. Assume that there exists $\delta\in K$ such that $\la H_L,\delta^{-1}\phi^{-1}H_L\phi\delta\ra$ admits a complete block system with blocks of size $7$. Now, it follows by~\cite{Dobson2010a} (see the two paragraphs following the proof of Corollary~\ref{coro1}) that $H_L$ and $\delta^{-1}\phi^{-1}H_L\phi\delta$ are conjugate in $\la H_L,\delta^{-1}\phi^{-1}H_L\phi\delta\ra^{(3)}$. Since $\la H_L,\delta^{-1}\phi^{-1}H_L\phi\delta\ra^{(3)}\le \la H_L,\delta^{-1}\phi^{-1}H_L\phi\delta\ra^{(2)}$, the corollary follows from Lemma~\ref{CItool} (and from Definition~\ref{defin1}).

Assume that there exists no $\delta\in K$ such that $\la H_L,\delta^{-1}\phi^{-1}H_L\phi\delta\ra$ admits a complete block system with blocks of size $7$. By Lemma~\ref{57tool}, the group $K$ admits a complete block system ${\cal C}$ with blocks of size $8$ and $\fix_K({\cal C})\vert_C$ is isomorphic to a primitive subgroup of $\AGL(3,2)$, for $C\in {\cal C}$.  Suppose that $7$ and $\vert\fix_K({\cal C})\vert$ are relatively prime. So, a Sylow $7$-subgroup of $K$ has order $7$.  We are now in the position to apply the argument in the proof of Corollary~\ref{5tool}. Let $\la\rho\ra$ be the subgroup of $H_L$ of order $7$. Then $\phi^{-1}\la\rho\ra\phi$ is a Sylow $7$-subgroup of $K$, and by a Sylow theorem there exists $\delta\in K$ such that $\delta^{-1}\phi^{-1}\la\rho\ra\phi\delta = \la\rho\ra$.  We then have that $\la H_L,\delta^{-1}\phi^{-1}H_L\phi\delta\ra$ has a unique Sylow $7$-subgroup, whose orbits form a complete block system with blocks of size $7$, contradicting our hypothesis on $K$.
We thus assume that $7$ divides $|\fix_K({\cal C})|$ and so $\fix_K({\cal C})$ acts doubly-transitively on $C$, for $C\in {\cal C}$.

Fix $C\in {\cal C}$ and let $L$ be  the point-wise stabilizer of $C$ in  $\fix_K({\cal C})$. Assume that $L\not = 1$. Now, we compute $K^{(2)}$ and we deduce that $H_L$ and $\phi^{-1}H_L\phi$ are conjugate in $K^{(2)}$, from which the corollary will follow from Lemma~\ref{CItool}. As $L\tl \fix_K({\cal C})$, we have $L\vert_{C'}\tl \fix_K({\cal C})\vert_{C'}$, for every $C'\in{\cal C}$. As a nontrivial normal subgroup of a primitive group is transitive \cite[Theorem 8.8]{Wielandt1964}, either $L\vert_{C'}$ is transitive or $L\vert_{C'} = 1$.  Let $\Gamma$ be a Cayley color digraph on $H$ with $K^{(2)}=\Aut(\Gamma)$.  Let ${\cal C} = \{C_i:i\in\Z_7\}$ where $C_i = \{(x_1,x_2,x_3,i):x_1,x_2,x_3\in\Z_2\}$, and assume without loss of generality that $C = C_0$.  Suppose that there is an edge of color $\kappa$ from some vertex of $C_i$ to some vertex of $C_j$, where $i\not = j$.  Then there is an edge of color $\kappa$ from some vertex of $C_0$ to some vertex of $C_{j-i}$.  Additionally, $j - i$ generates $\Z_7$, so there is a smallest integer $s$ such that $L\vert_{C_{s(j-i)}} = 1$ while $L\vert_{C_{(s+1)(j-i)}}$ is transitive.  As there is an edge of color $\kappa$ from some vertex of $C_{s(j-i)}$ to some vertex of $C_{(s+1)(j-i)}$, we conclude that there is an edge of color $\kappa$ from every vertex of $C_{s(j-i)}$ to every vertex of $C_{(s+1)(j-i)}$.  This implies that there is an edge of color $\kappa$ from every vertex of $C_i$ to every vertex of $C_j$, and then $\Gamma$ is the wreath product of a Cayley color digraph $\Gamma_1$ on $\Z_7$ and a Cayley color digraph $\Gamma_2$ on $\Z_2^3$. Since $\fix_K({\cal C})$ is doubly-transitive on $C$, we have $\Aut(\Gamma_2)\cong \Sym(8)$. Therefore $K^{(2)}=\Aut(\Gamma_1)\wr\Aut(\Gamma_2)\cong \Aut(\Gamma_1)\wr\Sym(8)$.  By \cite[Corollary 6.8]{DobsonM2009} and Lemma \ref{CItool} $H_L$ and $\phi^{-1}H_L\phi$ are conjugate in $K^{(2)}$. We henceforth assume that $L = 1$, that is, $\fix_K({\cal C})$ acts faithfully on $C$, for each $C\in {\cal C}$.

Define an equivalence relation on $H$ by $h\equiv k$ if and only if $\Stab_{\fix_K({\cal C})}(h) = \Stab_{\fix_K({\cal C})}(k)$. The equivalence classes of $\equiv$ form a complete block system ${\cal D}$ for $K$. As $\fix_K({\cal C})\vert_C$ is primitive and not regular, each equivalence class of $\equiv$ contains at most one element from each block of ${\cal C}$.  We conclude that ${\cal D}$ either consists of $8$ blocks of size $7$ or each block is a singleton. Since we are assuming that $K$ has no block system with blocks of size $7$, we have that each block of ${\cal D}$ is a singleton.

Fix $C$ and $D$ in ${\cal C}$ with $C\neq D$ and $h\in C$. Now, $\Stab_{\fix_K({\cal C})}(h)$ is isomorphic to a subgroup of $\GL(3,2)$ and acts with no fixed points on $D$.  From~\cite[Appendix B]{DixonM1996}), we see that $\AGL(3,2)$ is the only doubly-transitive permutation group of degree $8$ whose point stabilizer admits a fixed-point-free action of degree $8$. Therefore $\fix_K({\cal C})\cong \AGL(3,2)$.   Additionally, $\Stab_{\fix_K({\cal C})}(h)\vert_D$ is transitive on $D$.

Suppose that $\Gamma$ is a color digraph with $K^{(2)} = \Aut(\Gamma)$ and suppose that there is an edge of color $\kappa$ from $h$ to $\ell\in E$, with $E\in{\cal C}$ and $E\neq D$.  Then $\Stab_{\fix_K({\cal C})}(h)\vert_E$ is transitive, and so there is an edge of color $\kappa$ from $h$ to every vertex of $E$.  As $\fix_K({\cal C})$ is transitive on both $C$ and $E$, we see that there is an edge of color $\kappa$ from every vertex of $C$ to every vertex of $D$.  We conclude that $\Gamma$ is a wreath product of two color digraphs $\Gamma_1$ and $\Gamma_2$, where $\Gamma_1$ is a Cayley color digraph on $\Z_7$ and $\Gamma_2$ is either complete or the complement of a complete graph, and $K^{(2)} = \Aut(\Gamma_1)\wr \Sym(8)$.  The result then follows by the same arguments as above.
\end{proof}

\begin{proof}[Proof of Corollary~B]
From Corollary~\ref{coro1} and Theorem~A, it suffices to show that $\Z_2^2\times \Z_7$ is a CI-group with respect to color ternary relational structures. As the transitive permutation groups of degree $28$ are readily available in GAP or Magma, it can be shown using a computer that $\Z_2^2\times \Z_7$ is a CI-group with respect to color ternary relational structures. (We note that a detailed analysis similar to the proof of Corollary~A for the group $\Z_2^3\times \Z_7$ also gives a proof of this theorem.)
\end{proof}

\section{Concluding remarks}
In the rest of this paper, we discuss the relevance of Theorem~A  to the study of CI-groups with respect to ternary relational structures.
Using the software packages~\cite{Magma} and~\cite{GAP}, we have determined
that $\Z_2^5$ is not a
 CI-group with respect to ternary relational structures.
Here we report an example witnessing this fact: the group $G$ has order $2048$, $V$ and $W$ are two {\em nonconjugate} elementary abelian regular subgroups of $G$, and $X=(\{1,\ldots,32\},E)$ is a ternary relational structure with $G=\Aut(X)$.
\begin{eqnarray*}
V&=&
\langle\mbox{\footnotesize{(1,2)(3,4)(5,6)(7,8)(9,10)(11,12)(13,14)(15,16)(17,18)(19,20)(21,22)(23,24)(25,26)(27,28)(29,30)(31,32)}},\\
&&\langle\mbox{\footnotesize{(1,3)(2,4)(5,7)(6,8)(9,11)(10,12)(13,
    15)(14,16)(17,19)(18,20)(21,23)(22,24)(25,27)(26,28)(29,31)(30,32)}},\\
&&\langle\mbox{\footnotesize{ (1,5)(2,6)(3,7)(4,8)(9,13)(10,14)(11,15)(12,16)(17,21)(18,22)(19,23)(20,24)(25,29)(26,30)(27,31)(28,32)}},\\
&&\langle\mbox{\footnotesize{ (1,9)(2,10)(3,11)(4,12)(5,13)(6,14)(7,15)(8,16)(17,25)(18,26)(19,27)(20,28)(21,29)(22,30)(23,31)(24,32)}},\\
&&\langle\mbox{\footnotesize{  (1,17)(2,18)(3,19)(4,20)(5,21)(6,22)(7,23)(8,24)(9,25)(10,26)(11,27)(12,28)(13,29)(14,30)(15,31)(16,32)}}\rangle,\\
W&=&\langle\mbox{\footnotesize{(1,2)(3,4)(5,6)(7,8)(9,10)(11,12)(13,14)(15,16)(17,18)(19,20)(21,22)(23,24)(25,26)(27,28)(29,30)(31,32)}},\\
&&\mbox{\footnotesize{ (1,3)(2,4)(5,7)(6,8)(9,11)(10,12)(13,
    15)(14,16)(17,20)(18,19)(21,24)(22,23)(25,28)(26,27)(29,32)(30,31)}},\\
&&\mbox{\footnotesize{ (1,5)(2,6)(3,7)(4,8)(9,14)(10,13)(11,16)(12,15)(17,22)(18,21)(19,24)(20,
    23)(25,29)(26,30)(27,31)(28,32)}},\\
&&\mbox{\footnotesize{ (1,9)(2,10)(3,11)(4,12)(5,14)(6,13)(7,16)(8,15)(17,27)(18,28)(19,25)(20,26)(21,32)(22,31)(23,30)(24,29)}},\\
&&\mbox{\footnotesize{  (1,17)(2,18)(3,20)(4,19)(5,22)(6,21)(7,23)(8,24)(9,27)(10,28)(11,26)(12,25)(13,32)(14,31)(15,29)(16,30)}}\rangle,\\
G&=&\langle V,W,\mbox{\footnotesize{(25,26)(27,28)(29,30)(31,32),(1,11)(2,12)(3,9)(4,10)(5,13)(6,
    14)(7,15)(8,16)(17,19)(18,20)(25,27)(26,28)}}\rangle,\\
E&=&\{g((1,3,9)),g((1,5,25)): g\in G\}.
\end{eqnarray*}

\begin{defin}{\rm
For a cyclic group $M=\langle g\rangle$ of order $m$ and a cyclic group $\langle z\rangle$ of order $2^d$, $d\geq 1$, we
denote by $D(m,2^d)$ the group $\langle z\rangle\ltimes M$ with $g^z=g^{-1}$.}
\end{defin}

Combining Theorem~A with \cite[Theorem 9]{Dobson2003}, \cite[Lemma 6]{Dobson2003}, the construction given in \cite{Spiga2008} and the previous
paragraph, we have the
following result which lists every group that can be a CI-group with respect
to ternary relational structures (although not every group on the list
needs to be a CI-group with respect to ternary relational structures).

\begin{thrm}\label{new}
If $G$ is a CI-group with respect to ternary relational structures,
then all Sylow subgroups of $G$ are of prime order or isomorphic to
$\Z_4$,  $\Z_2^d$, $1\le d\le 4$, or $Q_8$. Moreover, $G = U\times
V$, where $\gcd(\vert U\vert, \vert V\vert) = 1$, $U$ is cyclic of order
$n$, with $\gcd(n,\varphi(n)) = 1$, and $V$ is one of the following:
\begin{enumerate}
\item $\Z_2^d$, $1\le d\le 4$, $D(m,2)$, or $D(m,4)$, where $m$ is
odd and $\gcd(nm,\varphi(nm)) = 1$,
\item $\Z_4$, $Q_8$.
\end{enumerate}
\noindent Furthermore,
\begin{enumerate}
\item[(a)] if $V = \Z_4$, $Q_8$, or $D(m,4)$ and $p\mid
n$ is prime, then $4\mathrel{\not|}(p - 1)$,
\item[(b)] if $V = \Z_2^d$, $d\ge 2$, or $Q_8$, then $3\not\vert\ n$,
\item[(c)] if $V = \Z_2^d$, $d\ge 3$, then $7\not\vert\ n$,
\item[(d)] if $V = \Z_2^4$, then $5\not\vert\ n$.
\end{enumerate}
\end{thrm}

\providecommand{\bysame}{\leavevmode\hbox to3em{\hrulefill}\thinspace}
\providecommand{\MR}{\relax\ifhmode\unskip\space\fi MR }
\providecommand{\MRhref}[2]{%
  \href{http://www.ams.org/mathscinet-getitem?mr=#1}{#2}
}
\providecommand{\href}[2]{#2}

\end{document}